\documentclass[12pt]{amsart}

\usepackage{amssymb}
\usepackage{bm}
\usepackage{graphicx}
\usepackage{psfrag}
\usepackage[usenames,dvipsnames]{xcolor}
\usepackage{enumerate}
\usepackage{multirow}
\usepackage{url}

\usepackage{comment}

\usepackage{algpseudocode}

\usepackage{mathtools}

\usepackage{xy}
\input xy
\xyoption{all}

\newcommand{\excise}[1]{}

\newtheorem{thm}{Theorem}[section]
\newtheorem{lemma}[thm]{Lemma}

\newtheorem{cor}[thm]{Corollary}
\newtheorem{prop}[thm]{Proposition}
\newtheorem{conj}[thm]{Conjecture}

\newtheorem{prob}[thm]{Problem}

\theoremstyle{definition}
\newtheorem{alg}[thm]{Algorithm}
\newtheorem{example}[thm]{Example}
\newtheorem{remark}[thm]{Remark}
\newtheorem{defn}[thm]{Definition}

\newtheorem{notation}[thm]{Notation}

\numberwithin{equation}{section}

\renewcommand\>{\rangle}
\newcommand\<{\langle}

\newcommand\CC{\mathbb{C}}

\newcommand\RR{\mathbb{R}}

\newcommand\ZZ{\mathbb{Z}}

\newcommand\KV{\mathrm{KV}}

\def\bar#1{{\overline {#1}}}

\DeclareMathOperator\Aut{Aut} 
\DeclareMathOperator\vol{vol} 

\DeclareMathOperator\Ap{Ap} 

\begin{document}

\mbox{}
\title[Numerical semigroups, polyhedra, and posets I]{Numerical semigroups, polyhedra, and posets I:\ \\ the group cone}

\author[Kaplan]{Nathan Kaplan}
\address{Mathematics Department\\University of California Irvine\\Irvine, CA 92697}
\email{nckaplan@math.uci.edu}

\author[O'Neill]{Christopher O'Neill}
\address{Mathematics Department\\San Diego State University\\San Diego, CA 92182}
\email{cdoneill@sdsu.edu}

\date{\today}

\begin{abstract}
Several recent papers have explored families of rational polyhedra whose integer points are in bijection with certain families of numerical semigroups.  One such family, first introduced by Kunz, has integer points in bijection with numerical semigroups of fixed multiplicity, and another, introduced by Hellus and Waldi, has integer points corresponding to oversemigroups of numerical semigroups with two generators.  In this paper, we provide a combinatorial framework from which to study both families of polyhedra.  We introduce a new family of polyhedra called group cones, each constructed from some finite abelian group, from which both of the aforementioned families of polyhedra are directly determined but that are more natural to study from a standpoint of polyhedral geometry.  We prove that the faces of group cones are naturally indexed by a family of finite posets, 
and illustrate how this combinatorial data relates to semigroups living in the corresponding faces of the other two families of polyhedra.  
\end{abstract}

\maketitle


\section{Introduction}
\label{sec:intro}

Let $\ZZ_{\ge 0}$ denote the set of nonnegative integers.  A \emph{numerical semigroup} $S$ is a subset of $\ZZ_{\ge 0}$ that contains 0, is closed under addition, and has finite complement in~$\ZZ_{\ge 0}$ (the~final condition is equivalent to requiring the greatest common divisor of the elements of $S$ is $1$).  We often specify a numerical semigroup via \emph{generators}, writing
$$S = \<n_1, \ldots, n_k\> = \{a_1n_1 + \cdots + a_kn_k : a_1, \ldots, a_k \in \ZZ_{\ge 0}\}$$
for the numerical semigroup generated by $n_1, \ldots, n_k$.  It is known that any numerical semigroup has a unique generating set that is minimal with respect to containment; its cardinality is known as the \emph{embedding dimension} of $S$.  
The \emph{set of gaps} of~$S$, denoted~$G(S)$, is the finite set $\ZZ_{\ge 0} \setminus S$.  The~cardinality of this set is the \emph{genus}~$g(S)$ of~$S$.  The smallest positive element of $S$ is called the \emph{multiplicity} of $S$, denoted $\mathsf m(S)$.  

To motivate the contents of this manuscript, we survey two counting problems involving numerical semigroups.  Each problem can be realized as an integer point counting problem in some family of rational polyhedra.  One of the primary insights of this manuscript is to identify a strong combinatorial connection between these polyhedra.  


\subsection{Counting by multiplicity and genus}
There has been much recent interest in counting numerical semigroups by genus; see the survey article \cite{kaplancounting} for an overview of problems and results in this area.  Let $N(g)$ denote the number of numerical semigroups $S$ with $g(S) = g$.  Bras-Amor\'os computed the first $50$ values of $N(g)$ and made several conjectures about the behavior of this function~\cite{brasamorosfibonacci}.  Zhai determined the asymptotic growth rate of $N(g)$, thereby proving $N(g) \le N(g+1)$ for~$g$ sufficiently large~\cite{zhaifibonacci}.  However, the following conjecture remains open, as the smallest upper bound for possible exceptions is well out of range of computation~\cite[Section~6]{ODorney}.

\begin{conj}[Bras-Am\'oros]\label{conj:conjNg}
For all $g \ge 1$, we have $N(g) \le N(g+1)$.
\end{conj}

One approach to Conjecture \ref{conj:conjNg} is to study a more refined counting problem.  Let $N_m(g)$ denote the number of numerical semigroups $S$ with $\mathsf m(S) = m$ and $g(S) = g$. See \cite[Table~1]{kaplannmg} for some values of $N_m(g)$.  

\begin{conj}[{\cite[Conjecture~7]{kaplannmg}}]\label{conj:nmgnondec}
For any $m \ge 2$ and $g \ge 1$, $N_m(g) \le N_m(g+1)$.
\end{conj}

In order to state what is known about $N_m(g)$ we introduce some additional notation.  A \emph{quasipolynomial} of degree $d$ is a function $f \colon \ZZ_{\ge 0} \to \CC$ of the form
\[
f(n) = c_d(n) n^d + c_{d-1}(n) n^{d-1} + \cdots + c_0(n)
\]
with periodic functions $c_i$ having integer periods, $c_d \neq 0$.  The function $c_d$ is called the \emph{leading coefficient} of $f$.  

\begin{thm}[{\cite[Proposition~7]{kaplannmg}, \cite[Theorem~4]{alhajjarkunz}}]\label{t:nmgquasi}
Fix $m \in \ZZ_{\ge 2}$.
\begin{enumerate}[(a)]
\item 
There is a quasipolynomial $p_m(g)$ of degree $m-2$ such that $N_m(g) = p_m(g)$ for all sufficiently large $g$.  

\item 
The leading coefficient of $p_m(g)$ is constant.

\end{enumerate}
\end{thm}

The proof of Theorem~\ref{t:nmgquasi} uses Ehrhart theory and a bijection between numerical semigroups of multiplicity $m$ and certain integer points in a rational polyhedron $P_m$, called the \emph{Kunz polyhedron} (we defer the formal definition to Section~\ref{sec:kunz}).  In particular, this yields a geometric interpretation of the leading coefficient of $p_m(g)$.  Additionally, the face structure of~$P_m$ was studied in~\cite{wilfmultiplicity} to provide a new approach to a longstanding conjecture of Wilf~\cite{wilfconjecture}.  


\subsection{Counting oversemigroups}
Let $S$ be a numerical semigroup.  An \emph{oversemigroup} of $S$ is a numerical semigroup $T$ with $T \supseteq S$.  Let $o(S)$ denote the number of oversemigroups of $S$.  Since $G(S)$ is a finite set, and any numerical semigroup $T \supseteq S$ has $G(T) \subseteq G(S)$ and is determined by its set of gaps, it is clear that $o(S)$ is finite.

Hellus and Waldi study the function $o(S)$ in the case $S = \<n, q\>$ is the smallest semigroup containing $n, q$ and $\gcd(n,q) = 1$.  For simplicity, write $o(n,q) = o(\<n, q\>)$.  
\begin{thm}[{\cite[Theorem 1.1]{oversemigroupcone}}]\label{t:HWthm}
Let $n \ge 2$ be a positive integer.

\begin{enumerate}[(a)]
\item 
There is a quasipolynomial $H_n(q)$ of degree $n - 1$ taking the value $o(n,q)$ at each positive integer $q$ relatively prime to $n$.

\item 
The leading coefficient $\lambda(n)$ of $H_n(q)$ is constant and
\[
\frac{1}{(n-1)! \cdot n!} \le \lambda(n) \le \frac{1}{(n-1) \cdot n!}.
\]

\end{enumerate}
\end{thm}

Like Theorem~\ref{t:nmgquasi}, the proof of Theorem~\ref{t:HWthm} uses a bijection between oversemigroups of $\<n, q\>$ and certain integer points in a rational polyhedral cone $O_n$, which we refer to here as the \emph{oversemigroup cone}.  The leading coefficient $\lambda(n)$ again has a geometric interpretation, as the volume of a particular cross section of $O_n$, from which the bounds in Theorem~\ref{t:HWthm}(b) are obtained in~\cite{oversemigroupcone}.  We defer the formal definition of $O_n$ to Section~\ref{sec:oversemigroupfaces}.  



\subsection{Enter the group cone}
The goal of this manuscript is to provide a common framework for studying the combinatorial structure of the Kunz polyhedron $P_m$ and the oversemigroup cone~$O_n$ via the introduction of a new family of polyhedra, the \emph{group cone} $\mathcal C(G)$ of a finite abelian group $G$ (Definition~\ref{d:groupcone}), which has also appeared in the context of discrete optimization as the cone of subadditive functions~\cite{subadditivitycone,groupmixedinteger}.  

\begin{itemize}
\item 
We give a combinatorial description of the faces of $\mathcal C(G)$ in terms of certain partially ordered sets.  In doing so, we complete the partial description of the faces of the Kunz polyhedron from~\cite{wilfmultiplicity}, and provide a previously unknown description of the faces of the oversemigroup cone.  In both settings, the poset corresponding to a face $F$ manifests within the divisibility posets of all semigroups lying in $F$.  

\item 
We identify particular cross sections of the group cone whose relative volumes equal the leading coefficients of the quasipolynomials in Theorems~\ref{t:nmgquasi} and~\ref{t:HWthm}.  
This implies that obtaining a triangulation for the group cone, which is a common method for computing or bounding cross section volumes, simultaneously yields control over the leading coefficients of both quasipolynomials.  

\end{itemize}



Several subsequent papers have already made use of the framework established here.  First, one operation that has been studied extensively in the numerical semigroup literature is called \emph{gluing}, and in~\cite{kunzfaces2}, the gluing operation is realized geometrically via a collection of combinatorial embeddings of group cones.  Second, a~major technique for understanding a numerical semigroup $S$ is to study its \emph{minimal presentations}, which encode minimal algebraic relations among the generators of $S$, and in~\cite{kunzfaces3}, a combinatorial connection is given between the face of $P_m$ containing $S$ and the minimal presentations of $S$.  Said another way, the face lattice of the group cone gives a stratification of the set of all numerical semigroups, wherein the numerical semigroups within each stratum have minimal presentations of a particular combinatorial type.  Third, in forthcoming work, the ideas of this paper are used to develop a combinatorial recipe for specializing free resolutions of numerical semigroup algebras.  



After defining the necessary terminology from polyhedral geometry in Section~\ref{sec:background}, we introduce the group cone $\mathcal{C}(G)$ in Section~\ref{sec:groupcone} and study the combinatorial data associated to its faces.  Sections~\ref{sec:kunz} and~\ref{sec:oversemigroupfaces} contain formal definitions of the Kunz polyhedron $P_m$ and the oversemigroup cone $O_n$, respectively, and provide precise connections to the faces of the group cone.  Combining results in these sections gives a direct correspondence between the Kunz polyhedron and the oversemigroup cone.  In Section~\ref{sec:leadingcoeffs}, we reduce the task of obtaining the precise leading coefficients of the quasipolynomial formulas in the counting problems described above to that of finding a triangulation of the corresponding group cone.  We conclude with Section~\ref{sec:aperycompute}, wherein we present an improved algorithm for computing the Ap\'ery set of a numerical semigroup $S$.  
We also include an appendix with some computational data related to the quasipolynomial functions introduced above.

\section{Background}
\label{sec:background}

In this section, we provide the necessary definitions related to convex rational polyhedra and partially ordered sets.  For more thorough introductions, see \cite{ziegler} and~\cite{ec}.  

A \emph{partially ordered set} (or \emph{poset}, for short) is a set $Q$ equipped with a relation $\preceq$ (called a \emph{partial order}) that is reflexive, antisymmetric, and transitive.  Given $q, q' \in Q$, we write $q' \prec q$ whenever $q' \preceq q$ and $q \ne q'$.  We say $q$ \emph{covers} $q'$ if $q' \prec q$ and there does not exist $q'' \in Q$ with $q' \prec q'' \prec q$.  If $(Q, \preceq)$ has a unique minimal element $0 \in Q$, the \emph{atoms} of $Q$ are the elements that cover $0$.  Posets are often depicted using a \emph{Hasse diagram} in which the elements of $Q$ are drawn so that whenever $q$ covers $q'$, $q$ is drawn above $q'$ and an edge is drawn from $q$ down to $q'$.  See Figure~\ref{f:m6posets} for examples.  

A \emph{rational polyhedron} $P \subset \RR^d$ is the set of solutions to a finite list of linear inequalities with rational coefficients, that is, 
$$P = \{x \in \RR^d : Ax \le b\}$$
for some matrix $A$ and vector $b$.  
If none of the inequalities can be omitted without altering $P$, we call this list the \emph{$H$-description} or \emph{facet description} of $P$ (such a list of inequalities is unique up to reording and multiplying both sides by a positive constant).  The~inequalities appearing in the $H$-description of $P$ are called \emph{facet inequalities} of~$P$.  

Given a facet inequality $a_1x_1 + \cdots + a_dx_d \le b$ of $P$, the intersection of $P$ with the hyperplane $a_1x_1 + \cdots + a_dx_d = b$ is called a \emph{facet} of $P$.  
A \emph{face} $F$ of $P$ is a subset of $P$ equal to the intersection of some collection of facets of $P$.  The set of facets containing $F$ is called the \emph{$H$-description} or \emph{facet description} of $F$.  The \emph{dimension} of a face $F$ is the dimension $\dim(F)$ of the affine linear span of $F$.  We say $F$ is a \emph{vertex} if $\dim(F) = 0$, an \emph{edge} if $\dim(F) = 1$ and $F$ is bounded, a \emph{ray} (or an \emph{extremal ray}) if $\dim(F) = 1$ and $F$ is unbounded, and a \emph{ridge} if $\dim(F) = d - 2$.  

If the origin is the unique point lying in the intersection of all facets of $P$ (or, equivalently, if $b = 0$ in the $H$-description of $P$), then we call $P$ a \emph{pointed cone}.  Separately, we say $P$ is a \emph{polytope} if $P$ is bounded.  If $P$ is a pointed cone, then any face $F$ equals the nonnegative span of the rays of $P$ it contains, and if~$P$ is a polytope, then any face $F$ equals the convex hull of the set of vertices of $P$ it contains; in each case, we call this the \emph{V-description} of $F$.  

The set of faces of a polyhedron $P$ forms a poset under containment that is a \emph{lattice} (i.e., any two faces have a unique greatest common descendant and a unique least common ancestor) and is \emph{graded} by dimension (i.e., whenever $F$ covers $F'$, we have $\dim(F) = \dim(F') + 1$).  If $P$ is a polytope, then every face of $P$ equals the convex hull of some collection of vertices of $P$ and the intersection of some collection of facets of~$P$, meaning the face lattice is both \emph{atomic} and \emph{coatomic}.

\section{The group cone}
\label{sec:groupcone}


We begin by defining the group cone of a finite abelian group.

\begin{defn}\label{d:groupcone}
Fix a finite abelian group $(G, +)$ and let $m = |G|$.  The \emph{group cone} $\mathcal C(G)$ is the set of all points $x \in \RR_{\ge 0}^{m-1}$ satisfying
\begin{equation}\label{eq:groupconeineqs}
x_a + x_b \ge x_{a+b} \qquad \text{for} \qquad a, b \in G \setminus \{0\} \qquad \text{with} \qquad a + b \ne 0,
\end{equation}
where the coordinates are indexed by the nonzero elements of $G$.  
\end{defn}

\begin{lemma}
For any finite abelian group $G$, the group cone $\mathcal C(G)$ is full-dimensional.  If $m = |G| \ge 3$, then the inequalities in~\eqref{eq:groupconeineqs} comprise the $H$-description of $\mathcal C(G)$.  
\end{lemma}

\begin{proof}
The first claim follows from the fact that for each nonzero $a \in G$, the vector $v$ with $v_a = 2$ and $v_b = 1$ for $b \ne a$ lies in $\mathcal C(G)$, since $\RR^{m-1}$ is spanned by these vectors.  

For the second claim, suppose $m \ge 3$.  We first verify that for each $a \in G \setminus \{0\}$, the inequality $x_a \ge 0$ is redundant.  Let $k = |a|$.  The key is for any $x \in \mathcal C(G)$,
$$cx_a = x_a + (c-1)x_a \ge x_{2a} + (c-2)x_a \ge x_{3a} + (c-3)x_a \ge \cdots \ge x_{ca}$$
for any positive integer $c < k$.  As such, if $k \ge 3$, then
$$(k+1)x_a = 2x_a + (k-1)x_a \ge x_{2a} + x_{(k-1)a} \ge x_a,$$
while if $k = 2$, then there exists some $b \in G \setminus \{0,a\}$, so
$$x_b + 2x_a \ge x_{b+a} + x_a \ge x_b.$$
In either case, we obtain $x_a \ge 0$.  
Lastly, given $a, b \in G \setminus \{0\}$ with $a + b \ne 0$, we see the point $x \in \mathcal C(G)$ with $x_a = x_b = 2$, $x_{a+b} = 4$, and $x_c = 3$ for all $c \notin \{a, b, a + b\}$ satisfies every inequality in~\eqref{eq:groupconeineqs} strictly except $x_a + x_b \ge x_{a+b}$.  
\end{proof}



The remainder of this section is dedicated to a combinatorial interpretation of the face lattice of the group cone in terms of Kunz-balanced posets.

\begin{defn}\label{d:kunzbalanced}
Fix a finite abelian group $G$.  A \emph{Kunz-balanced poset on $G$} is a poset~$\preceq$ with ground set $G$ such that for all $a, b, \in G$, $a \preceq b$ implies $b - a \preceq b$.  
Since $b \preceq b$ implies $0 = b - b \preceq b$ for all $b \in G$, any Kunz-balanced poset has unique minimal element $0 \in G$.  
\end{defn}

Throughout the rest of this section, when we have a finite abelian group $G$ and a subgroup $H \subset G$, we write $\bar x$ for the image of $x$ in $G/H$.

\begin{thm}\label{t:facetdesc}
There is an injective function
$$F \longmapsto (H, \preceq)$$
sending each face $F$ of $\mathcal C(G)$ to a pair $(H, \preceq)$, where 
$$H = \{0\} \cup \{h \in G : x_h = 0 \text{ for all } x \in F\} \subset G$$ 
is a subgroup of $G$ and $\preceq$ is the Kunz-balanced poset on $G/H$ with minimal element $\bar 0$ and the property that $x_a + x_b = x_{a + b}$ is a facet equation for $F$ if and only if $\bar a \preceq \bar a + \bar b$.  
\end{thm}

\begin{proof}
For $a, b \in H \setminus \{0\}$ with $a + b \ne 0$, we have
\[
0 = x_a + x_b \ge x_{a + b},
\]
so $x_{a + b} = 0$ for all $x \in F$.  As such, $H$ is a subgroup of $G$.  Also, for all $x \in F$, 
\[
x_a
= x_a + |H| x_h
\ge x_{a + h} + (|H| - 1)x_h
\ge \cdots
\ge x_{a + |H|\cdot h}
= x_a
\]
for each $a \in G \setminus H$ and $h \in H$, meaning $x_a = x_b$ whenever $\bar a = \bar b \in G/H$.

Next, define a reflexive relation $\preceq$ on $G/H$ with unique minimal element $\bar 0$ such that for each nonzero $a, b \in G$ with $a + b \ne 0$, we have $\bar a \preceq \bar a + \bar b$ whenever $x_a + x_b = x_{a + b}$ for all $x \in F$ (note that~$\preceq$ is well defined by the last sentence of the preceding paragraph).  If~$x_a + x_b = x_{a + b}$ and $x_{a + b} + x_{-b} = x_a$, then $x_b = -x_{-b}$, and nonnegativity of $x$ implies $x_b = 0$.  We conclude $\preceq$ is antisymmetric.  Lastly, if distinct nonzero $a, b, c \in G$ satisfy $x_a + x_{b-a} = x_b$ and $x_b + x_{c - b} = x_c$, then 
\[
x_{c}
= x_{b} + x_{c - b}
= x_a + x_{b - a} + x_{c - b}
\ge x_a + x_{c -a}
\ge x_{c}
\]
implies $x_a + x_{c - a} = x_c$, so $\preceq$ is transitive and thus a partial order.  Since $\preceq$ is Kunz-balanced by construction, the proof is complete.  
\end{proof}

\begin{defn}\label{d:kunzposet}
Given a face $F \subset \mathcal C(G)$ corresponding to $(H,\preceq)$ under Theorem~\ref{t:facetdesc}, we call $H$ the \emph{Kunz subgroup} of $F$ and $\preceq$ the \emph{Kunz poset} of $F$.  
\end{defn}

\begin{example}\label{e:CZ4}
The group cone $\mathcal{C}(\ZZ_4) \subset \RR^3$ has $4$ rays and $4\ 2$-dimensional facets.  Figure~\ref{f:CZ4} depicts the Kunz-balanced poset corresponding to each of these eight proper nontrivial faces.  Notice that whenever two faces $F, F' \subset \mathcal C(G)$ satisfy $F \subset F'$, the Kunz subgroup of $F$ contains the Kunz subgroup of $F'$, and if these subgroups coincide, then the Kunz poset of $F$ refines the Kunz poset of $F'$.  In particular, the lower right ray has nontrivial Kunz subgroup since the posets of its facets have contradictory orderings (indeed, $1 \preceq 3$ in one while $3 \preceq 1$ in the other).  
\end{example}

\begin{figure}[t]
\begin{center}
\includegraphics[height=3.0in]{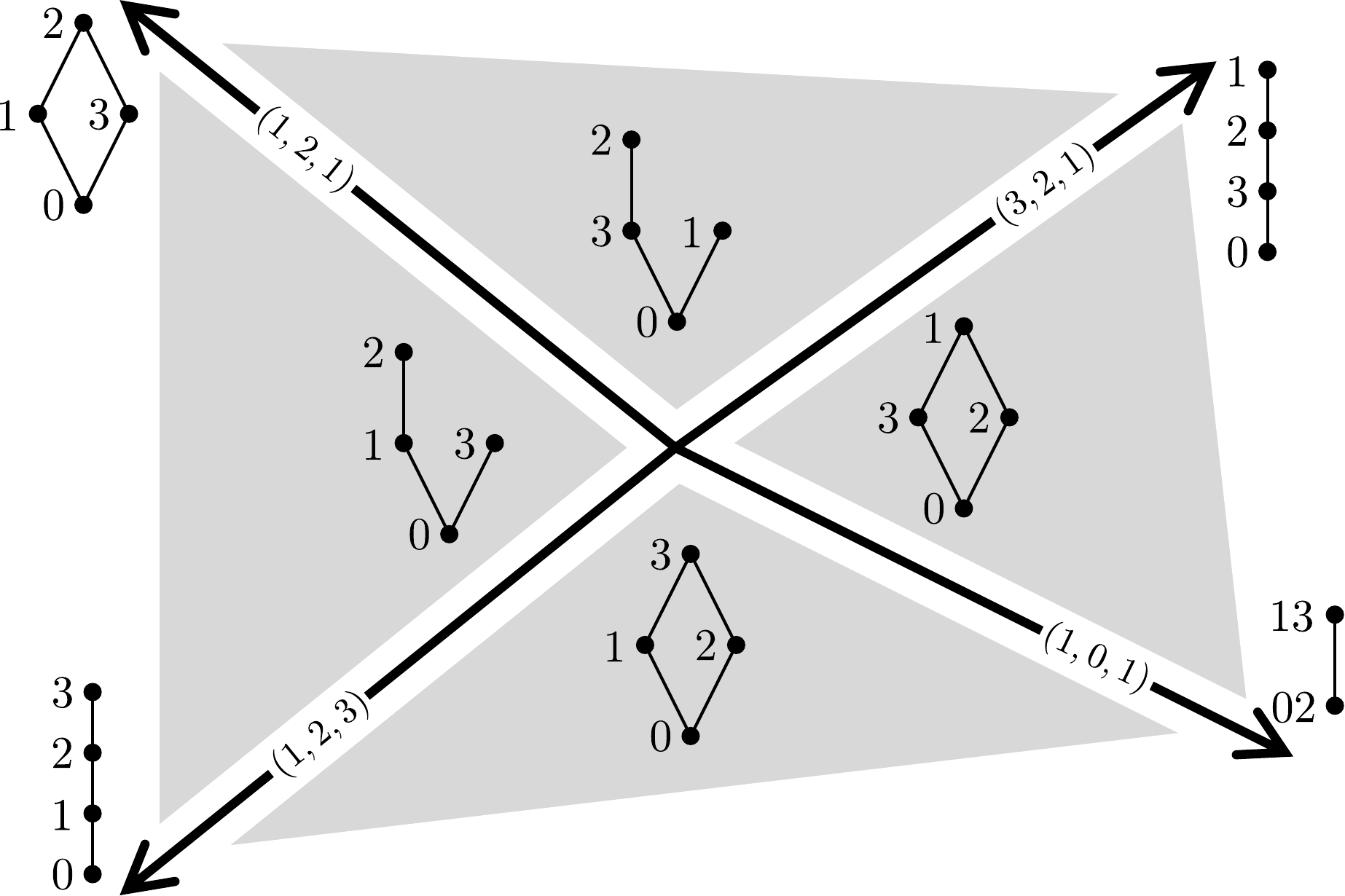}
\end{center}
\caption{The group cone $\mathcal{C}(\ZZ_4)$ with the Kunz poset of each proper, positive-dimensional face.}
\label{f:CZ4}
\end{figure}

The following is an immediate corollary of the proof of Theorem~\ref{t:facetdesc}.  

\begin{cor}\label{c:faceinjection}
For each subgroup $H \subset G$, the injection $\mathcal C(G/H) \hookrightarrow \mathcal C(G)$ given as $x \mapsto y$ with $y_a = x_{\bar a}$ for each $a \in G \setminus \{0\}$ induces an injection on face lattices.  
\end{cor}

\begin{remark}\label{r:automorphisms}
The automorphism group of $G$ acts on the group cone $\mathcal C(G)$ by permuting the coordinates of each $x \in \mathcal C(G)$, which induces an action on the face lattice of $\mathcal C(G)$ and thus on the associated Kunz posets under Theorem~\ref{t:facetdesc}.  In particular, a face is fixed by a particular automorphism of $G$ if and only if its Kunz poset is fixed as well.  
\end{remark}

\begin{example}\label{e:m6posets}
The cone $\mathcal C(\ZZ_6)$ has $11$ extremal rays, each of which is the nonnegative span of one of the following primitive integer vectors:
\begin{center}
$\begin{array}{l@{\qquad}l@{\qquad}l}
\begin{array}{l}
(1, 0, 1, 0, 1)
\\
(1, 2, 0, 1, 2)
\\
(2, 1, 0, 2, 1)
\end{array}
&
\begin{array}{l@{\quad}l}
(1, 2, 3, 4, 5)
&
(5, 4, 3, 2, 1)
\\
(1, 2, 3, 4, 2)
&
(2, 4, 3, 2, 1)
\\
(1, 2, 3, 1, 2)
&
(2, 1, 3, 2, 1).
\end{array}
&
\begin{array}{l}
(1, 2, 1, 2, 1)
\\
(1, 2, 3, 2, 1)
\\
\phantom{}
\end{array}
\end{array}$
\end{center}
The $3$ rays in the first column are those whose Kunz subgroup $H$ is nontrivial.  The ray $(1, 0, 1, 0, 1)$ has Kunz subgroup $H = \{0,2,4\}$ and is the image of the single ray generated by $(1)$ in $\mathcal C(\ZZ_2)$ under the embedding in Corollary~\ref{c:faceinjection}.  The other $2$ rays, namely $(1, 2, 0, 1, 2)$ and $(2, 1, 0, 2, 1)$, both have Kunz subgroup $H = \{0,3\}$ and are embeddings of the rays generated by $(1,2)$ and $(2,1)$ in $\mathcal C(\ZZ_3)$, respectively.  

The remaining 8 rays each have their Kunz poset on $\ZZ_6$ as depicted in Figure~\ref{f:m6posets}.  Each of the first 6 posets appears next to the poset lying in the same orbit under the action of the automorphism group of $\ZZ_6$ discussed in Remark~\ref{r:automorphisms}, 
and the last 2 are fixed by both automorphisms.  This is also visually evident in the Hasse diagrams of Figure~\ref{f:m6posets}, where the last 2 are again the ones fixed under both automorphisms.  
\end{example}

\begin{figure}[t!]
\begin{center}
\begin{tabular}{c@{\qquad}c@{\qquad}c@{\qquad}c}
\includegraphics[width=0.8in]{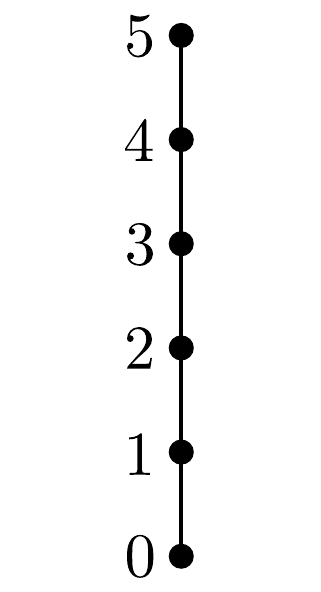}
&
\includegraphics[width=0.8in]{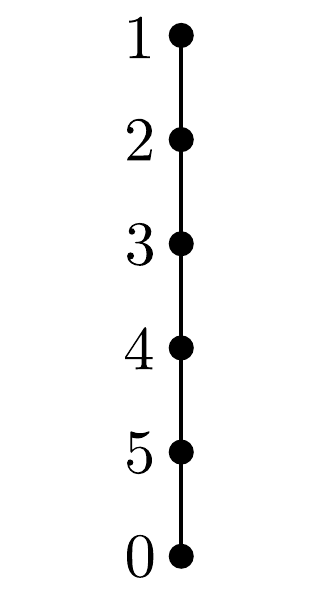}
&
\includegraphics[width=0.8in]{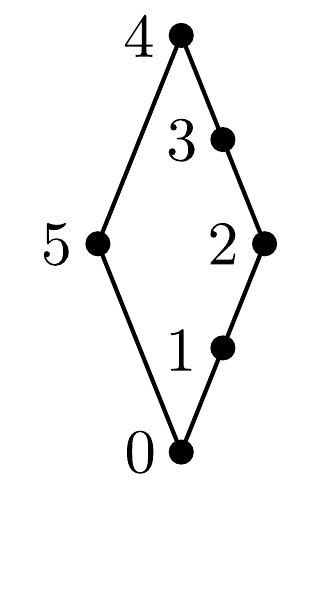}
&
\includegraphics[width=0.8in]{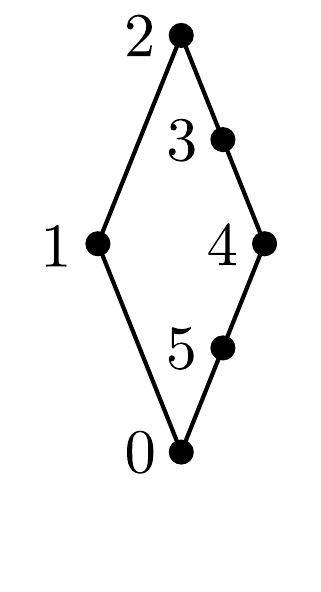}
\\[0.2em]
$(1,2,3,4,5)$
&
$(5,4,3,2,1)$
&
$(1,2,3,4,2)$
&
$(2,4,3,2,1)$
\\[1.0em]
\includegraphics[width=0.8in]{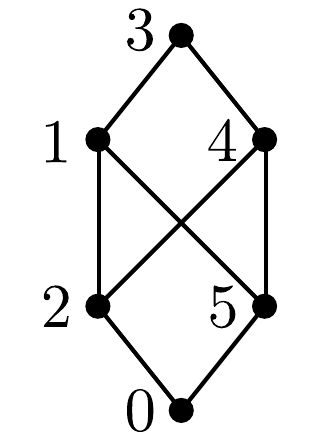}
&
\includegraphics[width=0.8in]{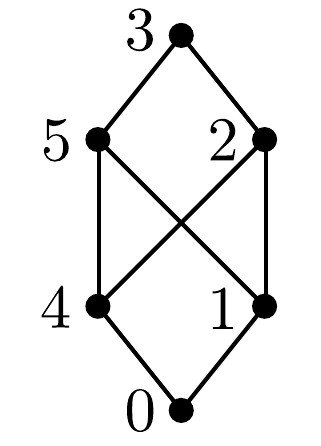}
&
\includegraphics[width=0.8in]{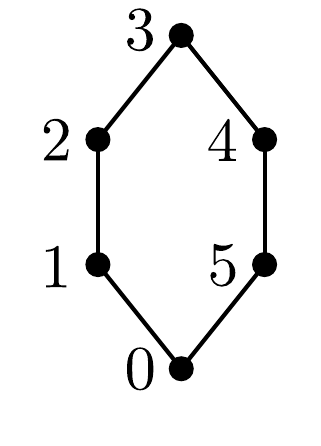}
&
\includegraphics[width=0.8in]{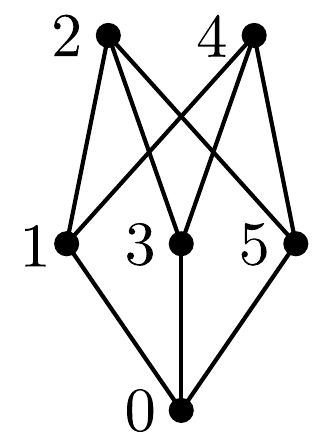}
\\[0.2em]
$(2,1,3,2,1)$
&
$(1,2,3,1,2)$
&
$(1,2,3,2,1)$
&
$(1,2,1,2,1)$
\end{tabular}
\end{center}
\caption{The Kunz posets of the extremal rays of $\mathcal C(\ZZ_6)$ in Example~\ref{e:m6posets}.}
\label{f:m6posets}
\end{figure}

\begin{prop}\label{p:minelements}
Fix a face $F$ of $\mathcal C(G)$ with Kunz subgroup $H$ and poset $(G/H,\preceq)$, and fix $a, b \in G/H$.  Let $M \subset G/H$ denote the set of atoms of $\preceq$.  
\begin{enumerate}[(a)]
\item 
If $a \prec b$, then $b$ covers $a$ if and only if $b - a \in M$.  In particular, each cover relation of $\preceq$ can be naturally labeled by an element of $M$.  

\item 
The coordinates of any point $x \in F$ are uniquely determined by the values of the coordinates $x_m$ for $m \in M$.  In particular, $\dim F \le |M|$.  

\end{enumerate}
\end{prop}

\begin{proof}
Since part~(a) depends only on the poset structure of $\preceq$, and the injection in Corollary~\ref{c:faceinjection} preserves dimension, it suffices to assume $H = \{0\}$ in both parts.  

Suppose $a \prec b$, so that $x_b = x_a + x_{b - a}$.  If $a \prec c \prec b$, then 
\[
x_a + x_{b - a} = x_b = x_c + x_{b - c} = x_a + x_{c - a} + x_{b - c}
\]
meaning $x_{b - a} = x_{c - a} + x_{b - c}$.  This means $c - a \prec b - a$, so $b - a \notin M$.  

Conversely, if $a \prec b$ but $b - a \notin M$, then some nonzero element $c$ satisfies $c \prec b - a$.  By transitivity, $c \prec b - a \prec b$, so
\[
x_{b - c} + x_c = x_b = x_a + x_{b - a} = x_a + x_c + x_{b - a - c},
\]
from which we conclude $a \prec b - c \prec b$.  

Next, suppose $b \in G$ is nonzero and not an atom, so that $b$ covers some other element $a \in G$.  By~part~(a), $b = a + m$ for some $m \in M$, and $x_b = x_a + x_m$, so part~(b) now follows from induction on the height of $b$ in $\preceq$.  
\end{proof}

\begin{remark}\label{r:dimensionbounds}
The inequality on $\dim F$ in Proposition~\ref{p:minelements}(b) can be strict, as demonstrated by $6$ of the $8$ posets in Example~\ref{e:m6posets}.  Also, the number of maximal elements of a Kunz poset need not bound the dimension of its corresponding face in the group cone in either direction.  Indeed, the poset corresponding to the ray $(1,2,1,2,1)$ of $\mathcal{C}(\ZZ_6)$ in Example~\ref{e:m6posets} has $2$ maximal elements, and the $2$-dimensional facet in $\mathcal C(\ZZ_4)$ with defining equation $x_3 = x_1 + x_2$ has only $1$ maximal element.  
\end{remark}

The following two propositions demonstrate that not every Kunz-balanced poset on an abelian group $G$ corresponds to a face of $\mathcal C(G)$.  They also provide evidence that characterizing precisely the set of Kunz-balanced posets that correspond to faces is likely difficult in general; see Remark~\ref{r:otherrestrictions}.  

\begin{prop}\label{p:diamondproperty}
Fix a face $F \subset \mathcal C(G)$ with Kunz poset $\preceq$ and trivial Kunz subgroup.  If $a, b, c \in G \setminus \{0\}$ satisfy $a \prec a + b$ and $a + b \prec a + b + c$, then $a \prec a + c$ and $a + c \prec a + b + c$.  
\end{prop}

\begin{proof}
Under the given assumptions, $a + b$ and $a + b + c$ are nonzero since neither is minimal under $\preceq$.  Additionally, $a + c$ must be nonzero, since otherwise 
$$b \prec a + b \prec a + b + c = b,$$
which is impossible.  Any $x \in F$ has $x_a + x_b = x_{a+b}$ and $x_{a+b} + x_c =  x_{a+b+c}$, so
$$x_{a+b+c} = x_a + x_b + x_c \ge x_{a+c} + x_b \ge x_{a+b+c}.$$
This implies $x_a + x_c = x_{a+c}$ and $x_{a+c} + x_b = x_{a+b+c}$, as desired.  
\end{proof}

\begin{remark}\label{r:diamondproperty}
Proposition~\ref{p:diamondproperty} is a kind of ``diamond property'' that is 
a reflection of the commutativity of $G$; see the left graphic in Figure~\ref{f:posetproperties} for a depiction.  
\end{remark}

\begin{prop}\label{p:cycleproperty}
Fix a face $F \subset \mathcal C(G)$ with Kunz poset $\preceq$ and trivial Kunz subgroup, a subgroup $G' \subset G$ with $|G'|$ odd, and $a \in G \setminus G'$ with $2a \notin G'$.  The following are equivalent:
\begin{enumerate}[(a)]
\item 
for some $b \in G'$, we have $i \prec 2i + b$ for every $i \in a + G'$; and

\item 
$i \prec j$ for each $i \in a + G'$ and $j \in 2a + G'$.  

\end{enumerate}
\end{prop}

\begin{proof}
The condition $2a \notin G'$ ensures every element of $G$ in both statements is nonzero.  For any $b \in G'$, if $g, g' \in G'$ satisfy $2(a + g) + b = 2(a + g') + b$, then the order of $g - g'$ in $G$ divides $2$, and since $|G'|$ is odd, we can conclude $g = g'$.  From this, we obtain
\begin{equation}\label{Gprimesum}
\sum_{g \in G'} 2x_{a + g} \ge \sum_{g \in G'} x_{2a + g}
\end{equation}
for all $x \in \mathcal C(G)$ by applying each inequality $x_i + x_{i + b} \ge x_{2i + b}$ exactly once for each $i \in a+G'$.  As such, if~(a) holds for some $b \in G'$, then equality holds in \eqref{Gprimesum}, so~(a) must hold for all $b \in G'$, from which~(b) follows.  Since~(b) clearly implies~(a), this completes the proof.  
\end{proof}



\begin{remark}\label{r:cycleproperty}

We give an example of Proposition~\ref{p:cycleproperty} where $G = \ZZ_9$, $G' = \{0, 3, 6\}$, and $a = 1$. 
This example is depicted in the middle of Figure~\ref{f:posetproperties}.  Doubling any element of $a + G' = \{1, 4, 7\}$ yields a distinct element of $2 + G' = \{2, 5, 8\}$, and this forces all possible relations $c \prec c'$ for $c \in 1 + G'$ and $c' \in 2 + G'$ to hold once the relations $1 \prec 2$, $4 \prec 8$, and $7 \prec 5$ hold. Intuitively, the sums of elements of $1 + G'$ are ``evenly distributed'' in $2 + G'$, so once sufficiently many relations between them are included, the rest must also appear.  
\end{remark}

\begin{remark}\label{r:otherrestrictions}
Propositions~\ref{p:diamondproperty} and~\ref{p:cycleproperty} are not the only restrictions on Kunz posets.  For example, when $G = \ZZ_8$, intersecting the facets $x_1 + x_5 = x_6$ and $x_3 + x_7 = x_2$ yields a face whose poset is depicted on the right in Figure~\ref{f:posetproperties}.  This is particularly noteworthy since it is an example of two facets with no variables in common whose intersection is not a ridge (a face of codimension $2$), a phenomenon that does not occur in $\mathcal C(\ZZ_m)$ for $m \le 7$.  In fact, the face of $\mathcal C(\ZZ_8)$ corresponding to this poset only has dimension $2$. 
\end{remark}

\begin{figure}[t!]
\begin{center}
\includegraphics[height=1.0in]{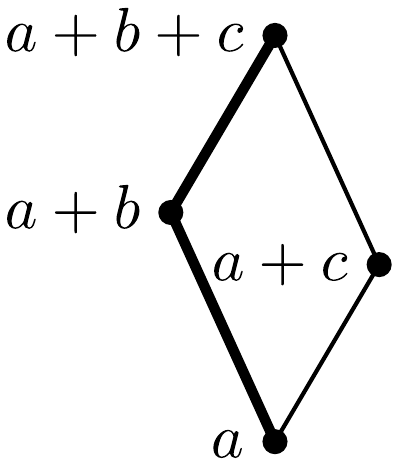}
\hspace{0.8in}
\includegraphics[height=1.0in]{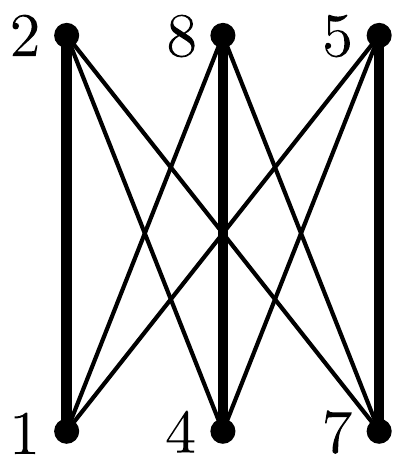}
\hspace{0.8in}
\includegraphics[height=1.0in]{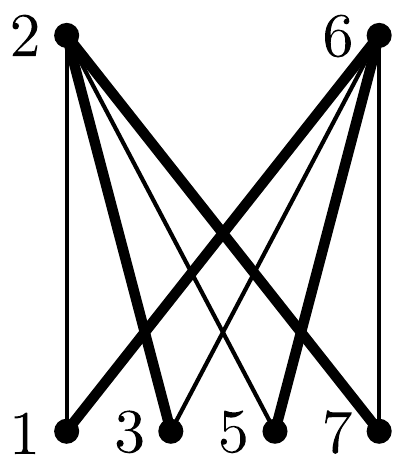}
\end{center}
\caption{In each of the depicted Hasse diagram excerpts, the thin lines are consequences of the thick lines.}
\label{f:posetproperties}
\end{figure}


\begin{prob}\label{prob:truebijection}
Determine when a given pair $(H, \preceq)$ of a subgroup $H \subset G$ and a Kunz-balanced poset $\preceq$ on $G/H$ corresponds to a face $F \subset \mathcal C(G)$.  
\end{prob}

\section{The Kunz polyhedron}
\label{sec:kunz}





We begin this section by defining the Kunz polyhedron $P_m$ and explaining the bijection between its integer points and numerical semigroups containing $m$.  Although many of the results in this section have appeared elsewhere, we state them here using the language of Section~\ref{sec:groupcone}.  Doing so answers \cite[Problem~3.14]{wilfmultiplicity} by providing a complete combinatorial characterization of the faces of $P_m$ (Theorem~\ref{t:kunzfaces}).  

One of the primary new insights of this section is Corollary~\ref{c:aperypoints}, which identifies a correspondence between integer points in $\mathcal{C}(\ZZ_m)$ and numerical semigroups containing $m$.  This correspondence first identifies these points of $\mathcal{C}(\ZZ_m)$ with integer points in $P_m$.  This allows one to move freely between the inequalities defining $P_m$ and those defining $\mathcal C(\ZZ_m)$.  This is often helpful when working with particular families of semigroups, since the homogeneity of the inequalities defining $\mathcal{C}(\ZZ_m)$ makes them easier to work with.


\begin{defn}
Let $S$ be a numerical semigroup.  The Ap\'ery set of $S$ with respect to an element $m \in S$ is the set
\[
\Ap(S;m) = \{s \in S \colon s-m \not\in S\}.
\]
It is easily shown $\Ap(S;m)$ has precisely $m$ elements, each in a distinct equivalence class modulo $m$.  More precisely, $\Ap(S;m) = \{0,a_1,\ldots, a_{m-1}\}$ where each $a_i \equiv i \bmod m$ is the smallest element of $S$ in its equivalence class modulo $m$.  For each $i$, we can write $a_i = k_i m + i$ for some $k_i \in \ZZ_{\ge 0}$.  The vector $(k_1,\ldots, k_{m-1})$ is called the \emph{Kunz coordinate vector of $S$ with respect to $m$}.  Let $\KV_m$ denote the function that takes a numerical semigroup containing $m$ to its Kunz coordinate vector with respect to $m$.
\end{defn}

It is easy to see that not every vector $(z_1,\ldots, z_{m-1}) \in \ZZ_{\ge 0}^{m-1}$ is the Kunz coordinate vector of a numerical semigroup containing $m$.  
The following set of linear inequalities that determine the image of $\KV_m$ in $\ZZ_{\ge 0}^{m-1}$ can be found in~\cite{kunz,kunzcoords}.  

\begin{defn}\label{d:kunzpolyhedron}
For $m \ge 2$, the \emph{Kunz polyhedron} $P_m \subset \RR^{m-1}$ is the set of points $(z_1, \ldots, z_{m-1})$ satisfying
\begin{align*}
z_i + z_j \ge z_{i+j}       &, \text{ for all } 1 \le i \le j \le m - 1 \text{ with } \ i + j < m, \\
z_i + z_j + 1 \ge z_{i+j-m} &, \text{ for all } 1 \le i \le j \le m - 1 \text{ with } \ i + j > m,
\end{align*}
and the \emph{strict Kunz polyhedron} $P'_m \subset \RR^{m-1}$ is given by $P_m' = P_m \cap \RR_{\ge 1}^{m-1}$.  
\end{defn}

\begin{remark}\label{r:nomorerelaxed}
The terminology used for $P_m$ and $P_m'$ varies across the literature.  It has often been called the ``Kunz polytope,'' although this conflicts with the conventions of polyhedral geometry, where ``polytopes'' are bounded polyhedra.  This was corrected in~\cite{wilfmultiplicity}, wherein $P_m$ and $P_m'$ were called the ``relaxed Kunz polyhedron'' and ``Kunz polyhedron'' respectively.  We believe the names in Definition~\ref{d:kunzpolyhedron} are the most appropriate, as (i) nonnegativity and positivity inequalities are frequently viewed as implicit or extra in the lattice point and integer optimization literature, and (ii) we will see below that  the relationship between numerical semigroups and the faces of $P_m$ is more direct than the connection to faces of $P_m'$, as $P_m'$ has several additional faces that come from the inequalities $z_i \ge 1$.
\end{remark}

\begin{thm}[{\cite{kunz, kunzcoords}}]\label{t:kunzbijection}
Let $m \ge 2$.  
\begin{enumerate}[(a)]
\item 
The map $\KV_m$ gives a bijection between numerical semigroups with multiplicity $m$ and integer points in $P'_m$.

\item 
The map $\KV_m$ gives a bijection between numerical semigroups containing $m$ and integer points in $P_m$.
\end{enumerate}
\end{thm}

\begin{notation}\label{n:semigroupinface0}
Given a numerical semigroup $S$ and a face $F \subset P_m$, we write $S \in F$ and say ``$S$ is in the face $F$'' to mean the Kunz coordinates of $S$ lie in $F$, that is, $\KV_m(S) \in F$. 
\end{notation}


In what follows, we show that the Kunz polyhedron $P_m$ is a translation of the group cone $\mathcal{C}(\ZZ_m)$, inducing a correspondence between their faces. 

\begin{defn}
Let $S$ be a numerical semigroup containing $m$ with 
\[
\Ap(S;m) = \{0,a_1, \ldots, a_{m-1}\}
\]
so that $a_i \equiv i \bmod m$ for each $i$.  The \emph{Ap\'ery poset} of $S$ is the divisibility poset of $S$ restricted to $\Ap(S;m)$, that is, with $a_i$ preceeds $a_j$ whenever $a_j - a_i \in S$.  The \emph{Kunz poset} of $S$ is the poset
\[
\mathcal{A}(S;m) = (\ZZ_m, \preceq)
\]
defined by $i \prec j$ whenever $a_j - a_i \in S$.  Said another way, $\mathcal A(S;m)$ is the divisibility poset of $S$ where each element is labeled with its equivalence class modulo $m$.
\end{defn}

The following result is basically equivalent to \cite[Theorem 3.10]{wilfmultiplicity} stated in the language of Section \ref{sec:groupcone}.

\begin{thm}\label{t:kunzfaces}
The Kunz polyhedon $P_m$ is a translation of $\mathcal C(\ZZ_m)$.  Moreover, any numerical semigroup $S$ in the interior of a face $F$ of $P_m$ has Kunz poset $\mathcal A(S;m)$ equal to the Kunz poset of $F$.  
\end{thm}

\begin{proof}
As in the proof of \cite[Theorem~3.10]{wilfmultiplicity}, the translation $\mathcal C(\ZZ_m) \to P_m$ is given by
\[
x \mapsto x + (-\tfrac{1}{m}, \ldots, -\tfrac{m-1}{m}),
\]
a fact which can be readily checked by substituting into the defining inequalities of~$P_m$.  

For the second claim, note that if a face $F \subset \mathcal{C}(\ZZ_m)$ has nontrivial Kunz subgroup $H \subset \ZZ_m$, then some coordinate of the corresponding face $F'$ of $P_m$ must be negative throughout $F'$.  As such, any face $F'$ containing semigroups has trivial Kunz subgroup, and so the result now follows from \cite[Theorem~3.10]{wilfmultiplicity}.  
\end{proof}

The following gives a method to identify semigroups in the group cone $\mathcal C(\ZZ_m)$ directly.  

\begin{cor}\label{c:aperypoints}
Fix a point $(z_1, \ldots, z_{m-1}) \in \ZZ_{\ge 0}^{m-1}$, let $a_i = z_im + i$ for each $i$, and fix a face $F \subset P_m$.  The following are equivalent:
\begin{enumerate}[(i)]
\item 
$(z_1, \ldots, z_{m-1}) \in F$; and

\item 
$(a_1, \ldots, a_{m-1})$ lies in the face $F' \subset \mathcal C(\ZZ_m)$ corresponding to $F$.  

\end{enumerate}
In both cases, $\{0, a_1, \ldots, a_{m-1}\}$ is the Ap\'ery set of a numerical semigroup.  
\end{cor}

\begin{proof}
This follows immediately upon checking the appropriate facet equations.  
\end{proof}

\begin{notation}\label{n:semigroupinface}
In analogy with Notation~\ref{n:semigroupinface0}, we say $S$ lies in the face $F' \subset \mathcal C(\ZZ_m)$ corresponding to $F$ if $S$ satisfies the conditions of Corollary~\ref{c:aperypoints}.
\end{notation}

The action of $\Aut(G)$ on $\mathcal C(G)$ given by coordinate permutation induces an action on the face lattice of $\mathcal C(G)$, and consequently on the face lattice of $P_m$.  The following result implies the property ``has a numerical semigroup'' is preserved by this action.  

\begin{cor}\label{c:latticepointsymmetry}
A face of $P_m$ contains numerical semigroups if and only if every face in its orbit under the action of $(\ZZ_m)^* \cong \Aut(\ZZ_m)$ on $P_m$ contains numerical~semigroups.  
\end{cor}

\begin{proof}
Fix $g, h \in \ZZ_{\ge 1}$ with $gh \equiv 1 \bmod m$.  Suppose $S$ is a numerical semigroup in $F$ with Ap\'ery set $\Ap(S;m) = \{0, a_1, \ldots, a_{m-1}\}$.  By Corollary~\ref{c:aperypoints}, $a = (a_1, \ldots, a_{m-1})$ lies in the corresponding face of $\mathcal C(\ZZ_m)$.  Acting on $a$ by $\overline g \in \Aut(\ZZ_m)$ yields the point $a' = (a_{h}, a_{2h}, \ldots, a_{(m-1)h})$, and scaling $a'$ by $g$ yields $(ga_{h}, ga_{2h}, \ldots, ga_{(m-1)h})$ in the same face as $a'$, where all of the subscripts are taken modulo $m$.  Moreover, we see
$$ga_{ih} \equiv gih \equiv i \bmod m,$$
so Corollary~\ref{c:aperypoints} implies that $\{0, ga_{h}, ga_{2h}, \ldots, ga_{(m-1)h}\}$ is the Ap\'ery set of some numerical semigroup in the appropriate face of $P_m$.  
\end{proof}

Corollary~\ref{c:oversemigrouptokunz} implies that in classifying the possible posets $\mathcal A(S;n)$ for fixed $n$, it suffices to consider semigroups with $\mathsf m(S) = n$. 

\begin{cor}\label{c:oversemigrouptokunz}
Given a numerical semigroup $S$ and any element $n \in S$, there exists a numerical semigroup $T$ with $\mathsf m(T) = n$ and $\mathcal A(T;n) = \mathcal A(S;n)$.  
\end{cor}

\begin{proof}
This follows from Theorem~\ref{t:kunzbijection}, as the vector difference of $\KV_n(S)$ and the vertex of $P_n$ must have all positive entries, so adding a multiple of this difference to $\KV_n(S)$ yields an integer point with all positive entries.  
\end{proof}

\begin{remark}\label{r:nokunzpts}
There are two reasons why a face $F \subset P_m$ may fail to contain any points corresponding to numerical semigroups.  The first is that some coordinates are negative throughout $F$; such faces are fully characterized by Theorem~\ref{t:facetdesc}.  The second is that $F$ contains positive rational points but no integer points, a property that is also reflected in the corresponding Kunz poset; Example~\ref{e:m6kunz} demonstrates this.  
\end{remark}

\begin{example}\label{e:m6kunz}

Let $m = 6$.  The Kunz polyhedron $P_6$ is obtained by translating $\mathcal C(\ZZ_6)$ by the vector $v = (-\tfrac{1}{6}, - \tfrac{2}{6}, -\tfrac{3}{6}, -\tfrac{4}{6}, -\tfrac{5}{6})$, so Example~\ref{e:m6posets} implies $P_6$ has $11$ extremal rays.  Of these rays, only the $2$ with vector directions $(1,2,3,4,5)$ and $(5,4,3,2,1)$ contain integer points.  These correspond to the rays of $\mathcal C(\ZZ_6)$ whose Kunz posets are total orderings; see Figure~\ref{f:m6posets}.

The $3$ rays of $P_6$ that correspond to rays of $\mathcal C(\ZZ_6)$ that are not listed in Figure~\ref{f:m6posets} each have a coordinate that is always negative.  For example, in the face $v + r\RR_{\ge 0}$ with vector direction $r = (1,0,1,0,1)$, every vector has second coordinate $-\tfrac{2}{6}$.


The remaining $6$ rays of $P_6$ contain points with all positive entries but still do not contain integer points.  This can be verified numerically from the coordinates of the vector direction of each ray, but can also be verified by proving that the corresponding posets in Figure~\ref{f:m6posets} cannot occur as the Ap\'ery poset of a numerical semigroup.  
\begin{enumerate}[(i)]
\item 
The poset for the ray with vector direction $(1,2,3,4,2)$ has $2 \preceq 4$ and $5 \preceq 4$, meaning an Ap\'ery set $\{0, a_1, \ldots, a_5\}$ with this divisibility poset would have to satisfy $2a_2 = a_4 = 2a_5$. This is impossible since $a_2$ and $a_5$ are distinct modulo~$6$.  

\item 
The poset for the ray with vector direction $(1,2,3,1,2)$ poses a similar issue since both $1 \preceq 2$ and $4 \preceq 2$ must hold.  

\item 
The poset for the ray with vector direction $(1,2,3,2,1)$ cannot be the Ap\'ery poset of a numerical semigroup since such an Ap\'ery set $\{0, a_1, \ldots, a_5\}$ would have to satisfy $a_3 = a_1 + a_2 = 3a_1$ as well as $a_3 = a_4 + a_5 = 3a_5$. This is impossible since $a_1 \ne a_5$.  

\item 
Suppose that the poset for the ray with vector direction $(1,2,1,2,1)$ occurred as the Ap\'ery poset of a numerical semigroup with Ap\'ery set $\{0,a_1,a_2,a_3,a_4,a_5\}$.  Since we have $a_2 = 2a_1 = a_3 + a_5$, either $a_3 < a_1 < a_5$ or $a_5 < a_1 < a_3$.  In either case, it is impossible to have $a_4 = 2a_5 = a_1 + a_3$.  This gives a contradiction.
\end{enumerate}
Corollary~\ref{c:latticepointsymmetry} implies the remaining $2$ rays of $P_6$ also contain no integer points.  
\end{example}

\begin{prob}\label{prob:nokunzpts}
Characterize, in terms of its Kunz poset, the conditions under which a given face $F$ contains numerical semigroups.  
\end{prob}

Throughout this section, we only utilize group cones $\mathcal C(G)$ for cyclic $G$.  

\begin{prob}\label{q:noncyclic}
Is there an analogue of Corollary~\ref{c:aperypoints} for $\mathcal C(G)$ when $G$ is not necessarily cyclic?  In particular, is there some family of semigroups in natural bijection with the integer points in some translation of $\mathcal C(G)$?  
\end{prob}

\section{Faces of the oversemigroup cone}
\label{sec:oversemigroupfaces}



In this section, we give a definition of the oversemigroup cone $O_n \subset \RR^n$ and explain the correspondence between integer points of $O_n$ and numerical semigroups containing~$n$.  We then give a connection between $O_n$ and the group cone $\mathcal{C}(\ZZ_n)$. Combining this with the results of the previous section gives a correspondence between $O_n$ and the Kunz polyhedron $P_n$.

\begin{defn}
For $n \ge 2$, the \emph{oversemigroup cone} $O_n \subset \RR^n$ is the set of points $(y_1, \ldots, y_n)$ satisfying
\begin{align*}
y_i + y_j \le y_{i+j}       &, \text{ for all } 1 \le i \le j \le n - 1 \text{ with } \ i + j < n, \text{ and} \\
y_i + y_j \le y_{i+j-n} + y_n &, \text{ for all } 1 \le i \le j \le n - 1 \text{ with } \ i + j > n.
\end{align*}
We also set notation for a particular face of the cone $O_n$, defining $O_n' \subset \RR^n$ by
\[
O_n' = O_n \cap \{y_1 = 0\}.
\]
\end{defn}
\noindent Note that we can view $O_n'$ as a cone in $\RR^{n-1}$.

\begin{prop}[{\cite[Lemma~4.1]{oversemigroupcone}}]\label{p:oversemigrouppts}
Fix $q, n \ge 1$ with $\gcd(n,q) = 1$.  Every integer point $y \in O_n$ with $y_n = q$ naturally corresponds to a numerical semigroup $S$ containing $n$ and $q$ with Ap\'ery set
\[
\Ap(S;n) = \{q - ny_1, \, 2q - ny_2, \, \ldots, \, (n-1)q - ny_{n-1}, \, 0\}.
\]
\end{prop}

Unlike the Kunz polyhedron, any numerical semigroup $S$ with $n \in S$ corresponds to infinitely many points in $O_n$, one for each $q \in S$ with $\gcd(n,q) = 1$.  Some of this redundancy is handled by the following proposition.  

\begin{prop}\label{p:nonaperyray}
Fix $n \ge 2$.  Every $y \in O_n$ can be written uniquely in the form
$$y = y' + y_1(1, 2, \ldots, n)$$
where $y' \in O_n$ and $y_1' = 0$.  Moreover, if $y$ corresponds to a numerical semigroup~$S$, then $y'$ also corresponds to $S$, and $y_n \in \Ap(S;n)$ if and only if $y_1 = 0$.  
\end{prop}

\begin{proof}
If $y \in O_n$, then $y' = y - y_1(1, 2, \ldots, n) \in O_n$, since for $i + j < n$ we have
\[
(y_i - iy_1) + (y_j - jy_1) = y_i + y_j - (i + j)y_1 \le y_{i+j} - (i + j)y_1
\]
and for $i + j > n$ we have
\[
(y_i - iy_1) + (y_j - jy_1) = y_i + y_j - (i + j)y_1 \le (y_{i+j-n} - (i + j - n)y_1) + (y_n - ny_1).
\]
This proves the first claim.  For the second claim, since 
$$\gcd(y_n', n) = \gcd(y_n - y_1n, n) = \gcd(y_n, n) = 1$$
we know $y'$ must correspond to some numerical semigroup $S'$ under Proposition~\ref{p:oversemigrouppts}.  Moreover, for each $i = 1, \ldots, n$, we have
$$iy_n' - ny_i' = i(y_n - ny_1) - n(y_i - iy_1) = iy_n - ny_i$$
so $\Ap(S';n) = \Ap(S;n)$ and thus $S = S'$.  The final claim follows from Proposition~\ref{p:oversemigrouppts}, since $y_n - y_1n \in \Ap(S;n)$ implies that $y_n \in \Ap(S;n)$ if and only if $y_1 = 0$.  
\end{proof}

By Proposition~\ref{p:nonaperyray}, $(1, 2, \ldots, n) \in O_n$ is the only ray with positive first coordinate.  Every face of $O_n$ that is not contained in $O_n'$ is simply the Minkowski sum of some face of $O_n'$ with the ray $(1, 2, \ldots, n)$.  This implies that in order to characterize the face lattice of $O_n$ it suffices to characterize the face lattice of $O_n'$.  
In the theorem below, we choose to think of $O_n'$ as a cone in $\RR^{n-1}$.

\begin{thm}\label{t:oversemigroupfaces}
For each $n \ge 2$, the linear map  $(y_2, \ldots, y_n) \mapsto (x_1, \ldots, x_{n-1})$ given by
\[
x_1 = \tfrac{1}{n}y_n, \qquad x_2 = \tfrac{2}{n}y_n - y_2, \qquad x_3 = \tfrac{3}{n}y_n - y_3, \qquad \ldots, \qquad x_{n-1} = \tfrac{n-1}{n}y_n - y_{n-1}
\]
maps $O_n'$ bijectively onto $\mathcal C(\ZZ_n)$.
Additionally, if a numerical semigroup $S$ corresponds to a point $y$ in the relative interior of some face $F \subset O_n'$, then applying the action of the automorphism $\sigma$ of $\ZZ_n$ with $\sigma(1) = y_n$ to each element of the ground set of $\mathcal A(S;n)$ yields the Kunz-balanced poset of the face of $\mathcal C(\ZZ_n)$ corresponding to $F$.  
\end{thm}

\begin{proof}
For clarity of notation we set $y_1 = 0$ for the rest of the proof.  We prove the first part of the statement by comparing the defining hyperplanes for $O_n'$ and for $\mathcal C(\ZZ_n)$.
For~each inequality $x_i + x_j \ge x_{i+j}$ of $\mathcal C(\ZZ_n)$ with $1 \le i \le j \le n-1$, there are 2 cases: 
\begin{enumerate}[(i)]
\item 
if $i + j < n$, then 
\[
\tfrac{i+j}{n}y_n - y_i - y_j = x_i + x_j \ge x_{i+j} = \tfrac{i+j}{n}y_n - y_{i+j},
\]
which simplifies to $y_i + y_j \le y_{i+j}$; and

\item 
if $i + j > n$, then
$$\tfrac{i+j}{n}y_n - y_i - y_j = x_i + x_j \ge x_{i+j} = x_{i+j-n} = \tfrac{i+j-n}{n}y_n - y_{i+j-n},$$
which simplifies to $y_i + y_j \le y_{i+j-n} + y_n$. 

\end{enumerate}
Since $\mathcal C(\ZZ_n)$ and $O_n'$ are both full dimensional cones in $\RR^{n-1}$, this proves the first claim.  

Next, suppose $S$ corresponds to a point $y$ in the relative interior of a face $F$ of $O_n'$.  Under the above transformation, Proposition~\ref{p:oversemigrouppts} implies $\Ap(S;n) = \{0, nx_1, \ldots, nx_{n-1}\}$, (note that the elements may not be written in order modulo $n$).   Elements of $\Ap(S;n)$ are distinct modulo $n$, so each $nx_i$ must be positive.  This means the Kunz subgroup of the image of $F$ in  $\mathcal C(\ZZ_n)$ is trivial.  However, $nx_1 = y_n$, so applying $\sigma$ to each of $nx_1, \ldots, nx_k$ in the Kunz poset $Q$ of $S$ yields the Kunz poset corresponding to the image of $F$ in $\mathcal C(\ZZ_n)$.  In particular, each facet equation $x_i + x_j = x_{i+j}$ of $F$ indicates divisibility as elements of $\Ap(S;n)$, so we obtain $\sigma^{-1}(i) \preceq \sigma^{-1}(i + j)$ in $Q$.  
\end{proof}


\begin{remark}\label{r:oversemigrouptokunz}
Composing the maps in Theorem~\ref{t:oversemigroupfaces} and Theorem~\ref{t:kunzfaces} pairs each face $F \subset O_n'$ with a face $F'$ of the Kunz polyhedron $P_n$.  We see that $F$ contains an integer point corresponding to a numerical semigroup if and only if $F'$ does, even though integer points of $F$ do not necessarily get sent to integer points of $F'$ under this composition.

\end{remark}


\begin{example}\label{e:m6oversemigroup}
The cone $O_6 \subset \RR^6$ has 12 extremal rays.  One is the span of $(1, 2, 3, 4, 5, 6)$, as described in Proposition~\ref{p:nonaperyray}.  The remaining 11 rays, whose sum equals $O_n'$, each equal the nonnegative span of one of the following primitive integer vectors:
\begin{center}
$\begin{array}{l@{\qquad}l@{\qquad}l}
\begin{array}{l}
(0, 1, 1, 2, 2, 3)
\\
(0, 0, 1, 1, 1, 2)
\\
(0, 1, 2, 2, 3, 4)
\end{array}
&
\begin{array}{l@{\quad}l}
(0, 0, 0, 0, 0, 1)
&
(0, 1, 2, 3, 4, 5)
\\
(0, 0, 0, 0, 1, 2)
&
(0, 0, 1, 2, 3, 4)
\\
(0, 0, 0, 1, 1, 2)
&
(0, 1, 1, 2, 3, 4).
\end{array}
&
\begin{array}{l}
(0, 0, 1, 1, 2, 3)
\\
(0, 0, 0, 1, 2, 3)
\\
\phantom{}
\end{array}
\end{array}$
\end{center}
Each ray above corresponds to the analogously positioned vector in Example~\ref{e:m6posets} after truncating the initial $0$ coordinate and applying the bijection in Theorem~\ref{t:oversemigroupfaces}.  The only two that contain integer points corresponding to numerical semigroups are those generated by $(0, 0, 0, 0, 0, 1)$ and $(0, 1, 2, 3, 4, 5)$, since these are the only ones with last coordinate relatively prime to $6$.  
\end{example}

Much of the structure highlighted in Example~\ref{e:m6oversemigroup} would not be readily clear without the explicit bijection in Theorem~\ref{t:oversemigroupfaces}, as the Ap\'ery set construction in Proposition~\ref{p:oversemigrouppts} breaks down for points whose last coordinate is not coprime to~$n$.

\section{Leading coefficients of Ehrhart quasipolynomials}
\label{sec:leadingcoeffs}

Recall that $N_m(g)$ equals the number of numerical semigroups with multiplicity $m$ and genus $g$, and $o(n,q)$ equals the number of numerical semigroups containing two relatively prime integers $n$ and $q$.  These functions coincide with quasipolynomials $p_m(g)$ (for $g \gg 0$) and $H_n(q)$ by Theorems~\ref{t:nmgquasi} and~\ref{t:HWthm}, respectively.  

The main result of this section is Theorem~\ref{t:leadingcoefficients}, which expresses the leading coefficients in the quasipolynomials $p_m(g)$ and $H_m(q)$ in terms of an arbitrary triangulation of the group cone $\mathcal{C}(\ZZ_m)$.  Finding an explicit triangulation is still open (Problem~\ref{prob:triangulation}), and will likely require first characterizing the extremal rays of $\mathcal C(\ZZ_m)$.  

Let $\gamma(m)$ denote the leading coefficient of $p_m(g)$.  
In what follows, given a subset~$P$ of Euclidean space whose affine linear span has dimension $d$, the \emph{relative volume} of~$P$, denoted $\vol(P)$, is the $d$-dimensional Euclidean volume of $P$ normalized with respect to the sublattice of the affine span of $P$.  

For each $m \ge 2$, \cite[Theorem~4]{alhajjarkunz} and Proposition~\ref{p:oversemigrouppts} imply the leading cofficients of $p_m(g)$ and $H_m(q)$ are given~by
\begin{align}
\label{eq:gammavolume}
\gamma(m) &= \vol(\mathcal C(\ZZ_m) \cap \{x \in \RR^{m-1} : x_1 + \cdots + x_{m-1} = 1\}) \text{ and}
\\
\label{eq:lambdavolume}
\lambda(m) &= \vol(O_m \cap \{x \in \RR^{m-1} : x_m = 1\}),
\end{align}
respectively, where $\vol(-)$ denotes relative volume and $\|\cdot\|_1$ denotes the $\ell_1$-norm.  
As~stated previously, both~\eqref{eq:gammavolume} and~\eqref{eq:lambdavolume} use Ehrhart's theorem~\cite{ehrhart}.  For a different perspective on the function $o(n,q)$ that exploits a bijection between oversemigroups of $\<n,q\>$ and integer points in a different polyhedron, see~\cite{chek}.




Fix $m \in \ZZ_{\ge 2}$ and a triangulation $\mathcal T$ of $\mathcal C(\ZZ_m)$.  For each simplicial cone $T \in \mathcal T$, write 
$$\begin{array}{r@{}c@{}l}
V(T)
&{}={}& \displaystyle \vol(T \cap \{x \in \RR^{m-1} : x_1 + \cdots + x_{m-1} = 1\}) \\
&{}={}& \displaystyle \frac{1}{(m-2)! \displaystyle \prod_i \|r_i\|_1}\left|\det\begin{pmatrix}
r_{1,1} & \cdots & r_{m-1,1} \\
\vdots & \ddots & \vdots \\
r_{1,m-1} & \cdots & r_{m-1,m-1} \\
\end{pmatrix}\right|
\end{array}$$
for the relative volume of $T \cap \{x \in \RR^{m-1} : x_1 + \cdots + x_{m-1} = 1\}$, 
where $r_1, \ldots, r_{m-1}$ are directional vectors of the rays of $T$.  By \cite[Theorem~4]{alhajjarkunz}, the leading coefficient of $p_m(g)$ equals
\[
\vol\big(\mathcal C(\ZZ_m) \cap \{x \in \RR^{m-1} : x_1 + \cdots + x_{m-1} = 1\}\big) = \sum_{T \in \mathcal T} V(T). 
\]

The leading coefficient $\lambda(m)$ of $H_m(q)$ is the relative volume of 
\[
Q = O_m \cap \{x \in \RR^m : x_m = 1\}.
\]
By Proposition~\ref{p:nonaperyray}, $Q$ is a pyramid over 
\[
Q' = O_m' \cap \{x \in \RR^{m-1} : x_m = 1\}
\]
with height $\tfrac{1}{m}$.  Under the linear map in Theorem~\ref{t:oversemigroupfaces}, the image of $Q'$ in $\mathcal C(\ZZ_m)$ is 
\[
Q'' = \mathcal C(\ZZ_m) \cap \{x_1 = \tfrac{1}{m}\},
\]
and combining factors from each of the above operations, we obtain 
\[
\vol(Q) = \frac{1}{m(m-1)} \vol(Q') = \frac{1}{m(m-1)} \vol(Q'') = \frac{1}{m^{m-1}(m-1)} \vol(mQ'').
\]
The final observation is that for each $T \in \mathcal T$,
we have
\[
V(T) = \vol(T \cap \{x \in \RR^{m-1} : x_1 = 1\}) \prod_{r_i \in T} \frac{r_{i,1}}{\|r_i\|_1}.
\]
Note that, as a consequence of Theorem~\ref{t:facetdesc}, every nonzero vector on a ray of $\mathcal C(\ZZ_m)$ has nonzero first coordinate.  
This proves the following.  

\begin{thm}\label{t:leadingcoefficients}
The leading coefficient of the quasipolynomial $p_m(g)$ is 
\[
\gamma(m) = \sum_{T \in \mathcal T} V(T),
\]
and the leading coefficient of the quasipolynomial $H_m(q)$ is 
\[
\lambda(m) = \frac{1}{m^{m-1}(m-1)} \sum_{T \in \mathcal T} V(T) \prod_{r_i \in T} \frac{\|r_i\|_1}{r_{i,1}}.
\]
\end{thm}

\begin{example}\label{e:leadingcoefficients}
Let $m = 4$.  One triangulation of $\mathcal C(\ZZ_4)$ consists of the cones
$$\begin{array}{r@{}c@{}l}
T_1 &{}={}& \RR_{\ge 0}(1,0,1) + \RR_{\ge 0}(1,2,1) + \RR_{\ge 0}(1,2,3) \qquad \text{and} \\
T_2 &{}={}& \RR_{\ge 0}(1,0,1) + \RR_{\ge 0}(1,2,1) + \RR_{\ge 0}(3,2,1),
\end{array}$$
which have relative volumes (as defined above)
$$
V(T_1) = \tfrac{1}{96}\left|\det\begin{pmatrix}
1 & 1 & 1 \\
0 & 2 & 2 \\
1 & 1 & 3 \\
\end{pmatrix}\right| = \tfrac{1}{24}
\qquad \text{and} \qquad
V(T_2) = \tfrac{1}{96}\left|\det\begin{pmatrix}
1 & 1 & 3 \\
0 & 2 & 2 \\
1 & 1 & 1 \\
\end{pmatrix}\right| = \tfrac{1}{24}.$$
As such, the leading coefficient of $p_4(g)$ is
$$V(T_1) + V(T_2) = \tfrac{1}{12}$$
and the leading coefficient of the $H_4(q)$ is
\[
\tfrac{1}{192}(48V(T_1) + 16V(T_2)) = \tfrac{1}{72},
\]
which agree with the computations in \cite{alhajjarkunz} and \cite{oversemigroupcone}, respectively.  
\end{example}

\begin{prob}\label{prob:triangulation}
For each finite abelian group $G$, find a triangulation of $\mathcal C(G)$.  
\end{prob}

\section{Computing the Ap\'ery set of a numerical semigroup}
\label{sec:aperycompute}

In the process of writing this paper, an improved implementation of the Ap\'ery set function was written for the 
\texttt{GAP} package \texttt{numericalsgps}~\cite{numericalsgpsgap}.  The original implementation, based on the circle-of-lights algorithm proposed by Wilf~\cite{wilfconjecture}, enumerates each positive integer, beginning at the multiplicity $m$, and stops when all $m-1$ positive elements $a_1, \ldots, a_{m-1}$ of the Ap\'ery set have been obtained.  The key idea is that one only needs to enumerate within equivalence classes modulo $m$ for which $a_i$ has not yet been found, and checking if a given integer $n \equiv i \bmod m$ lies in $\Ap(S,m)$ can be done by checking if $n = a_{j} + a_{i-j}$ for some previously obtained elements $a_j, a_{j-i} \in \Ap(S,m)$.  In~this sense, the circle-of-lights algorithm is in fact computing the Ap\'ery poset of $S$.  

Algorithm~\ref{a:aperyset}, in contrast, walks up the Kunz poset instead of the Ap\'ery poset.  The~Ap\'ery set elements are obtained starting with the bottom of the Kunz poset and using Proposition~\ref{p:minelements}(b) to check potential cover relations above each new element.  New elements are traversed in the order in which they are encountered, using a queue (first-in-first-out) data structure, and only the smallest element in each equivalence class modulo $m$ is retained (denoted as $a(0), \ldots, a(m-1)$ in the algorithm).  

The resulting implementation is particularly effective for numerical semigroups with
\begin{enumerate}[(i)]
\item 
``small'' embedding dimension, or

\item 
some generators that are much larger than the multiplicity (i.e., those represented by ``large points'' in the Kunz polyhedron), 

\end{enumerate}
as such semigroups can have long sequences of successive integers outside of $\Ap(S,m)$.  
We ran calculations for large numbers of randomly chosen numerical semigroups and include a representative sample comparing the runtimes of Algorithm~\ref{a:aperyset} and the original implementation in Table~\ref{tb:aperysetruntimes}.  

\begin{alg}\label{a:aperyset}
Computes the Ap\'ery set of a numerical semigroup from its generators.
\begin{algorithmic}
\Function{AperySetOfNumericalSemigroup}{$m, n_1, \ldots, n_k$}
\State Initialize a queue $Q \hookleftarrow 0$
\State $a(0) \gets 0$ and $a(i) \gets \infty$ for each $i = 1, \ldots, m-1$
\While{$|Q| > 0$}
	\State Dequeue $n \hookleftarrow Q$, disregarding any with $n > a(n \bmod m)$
	\ForAll{$g = n_1, \ldots, n_k$}
		\If{$n + g < a((n + g) \bmod m)$}
			\State $a((n + g) \bmod m) \gets n + g$
			\State Enqueue $Q \hookleftarrow n + g$
		\EndIf
	\EndFor
\EndWhile
\State \Return $\{0, a(1), \ldots, a(m-1)\}$
\EndFunction
\end{algorithmic}
\end{alg}

\begin{remark}\label{r:posetseverywhere}
Algorithm~\ref{a:aperyset} does not make any reference to the Ap\'ery or Kunz posets, but the underlying idea uses this poset structure.  This appears to be relatively common in the numerical semigroup literature, where other results utilize this additional structure without referring to it explicitly.  
\end{remark}

\begin{table}
\begin{tabular}{l|l|l}
$S$ & \texttt{GAP} \cite{numericalsgpsgap} & Algorithm~\ref{a:aperyset} \\
\hline
$\<1000, 1001\>$      & 90 ms   & 0 ms    \\
$\<10000, 10001\>$      & 6720 ms   & 10 ms    \\
$\<27143, 30949, 35207\>$      & 52250 ms   & 40 ms    \\
$\<50632, 225750, 249397, 468508\>$      & 176480 ms   & 140 ms    \\
\end{tabular}
\medskip
\caption{Runtimes for Ap\'ery set computations, each using \texttt{GAP} and the package \texttt{numericalsgps}~\cite{numericalsgpsgap}.}
\label{tb:aperysetruntimes}
\end{table}

\section*{Acknowledgements}

The authors would like to thank Jackson Autry, Winfried Bruns, Jes\'us De Loera, Hayan Nam, Sherilyn Tamagawa, and Dane Wilburne for numerous helpful conversations, as well as the anonymous referees for their helpful comments.  
The first author is supported by NSF Grant DMS-1802281.



\begin{table}
\begin{tabular}{l|l|l|l|l|l|}
$G$ & Leading coefficient            & Next coefficient              & Period   \\ 
\hline
$\ZZ_3$   & $1/3$                                                & -                                                       & $3$                                            \\[0.4em] 
$\ZZ_4$   & $1/12$                                               & -                                                       & $12$                                           \\[0.4em] 
$\ZZ_5$   & $1/135$                                              & $2/45$                                                  & $30$                                           \\[0.4em] 
$\ZZ_6$   & $71/(2^4 \cdot 3^6 \cdot 7)$                         & $(5\cdot 71)/(2^3\cdot 3^6 \cdot 7)$                    & $2^2 \cdot 3^2 \cdot 5 \cdot 7$                \\[0.4em] 
$\ZZ_7$   & $(23 \cdot 71)/(2^9 \cdot 3^4 \cdot 5^3 \cdot 7)$    & $(23\cdot 71)/(2^9\cdot 3^3\cdot 5^2\cdot 7)$           & $2^2\cdot 3\cdot 5\cdot 7$                     \\[0.4em] 
$\ZZ_8$   & $(113 \cdot 108461)/$                                & $(113 \cdot 108461)/$                                   & $2^4 \cdot 3^2 \cdot 5\cdot 7\cdot 11\cdot 13$ \\ 
     & $(2^{11} \cdot 3^7 \cdot 5^3 \cdot 7^2 \cdot 11\cdot 13)$ & $(2^{11} \cdot 3^6 \cdot 5^3 \cdot 7 \cdot 11\cdot 13)$ &                                                \\[0.4em]
$\ZZ_2^2$ & $1/8$                                                & -                                                       & $2$                                            \\[0.4em] 
$\ZZ_2^3$ & $1/(2^{11} \cdot 3 \cdot 5^2)$                       & $7 / (2^{11} \cdot 5^2)$                                & $20$                                           \\[0.4em] 
$\ZZ_3^2$ & $3001/(2^4 \cdot 3^9 \cdot 5^3 \cdot 7^2 \cdot 11)$  & $3001 / (2^2 \cdot 3^9 \cdot 5^3 \cdot 7 \cdot 11)$     & $3^3 \cdot 5 \cdot 7 \cdot 11$                 \\[0.4em] 
$\ZZ_4 \times \ZZ_2$ & $479 / (2^{11} \cdot 3^7 \cdot 5^2)$      & $(7 \cdot 479) / (2^{11} \cdot 3^6 \cdot 5^2)$          & $2^4 \cdot 3^2 \cdot 5$                        \\[0.4em] 

\end{tabular}
\medskip
\caption{Data for $L_m(g)$, obtained using \texttt{Normaliz} \cite{normaliz3}.}
\label{tb:groupconequasidata}
\end{table}

\begin{table}
\begin{tabular}{l|l|l|l|l|l|}
$m$ & Leading coefficient   & Next coefficient     & Period & Initial $g$ \\ 
\hline
$3$ & $1/3$           & -              & $3$    & $2$         \\[0.4em] 
$4$ & $1/12$          & $1/2$          & $6$    & $4$         \\[0.4em] 
$5$ & $1/135$         & $4/45$         & $30$   & $7$         \\[0.4em] 
$6$ & $71/(2^4\cdot 3^6 \cdot 7)$      & $(5\cdot 71)/(2^2 \cdot 3^6 \cdot 7)$    & $2^2 \cdot 3^2 \cdot 5 \cdot 7$ & $11$        \\[0.4em] 
$7$ & $(23\cdot 71)/(2^9 \cdot 3^4 \cdot 5^3 \cdot 7)$ & $(23\cdot 71)/(2^8\cdot 3^3 \cdot 5^2 \cdot 7)$ & $2^2 \cdot 3\cdot 5\cdot 7$  & $16$        \\ 
\end{tabular}
\medskip
\caption{Data for $p_m(g)$, obtained using \texttt{Normaliz} \cite{normaliz3}.}
\label{tb:kunzpolyhedronquasidata}
\end{table}

\begin{table}
\begin{tabular}{l|l|l|l|l|l|}
$m$ & Leading coefficient                   & Next coefficient                    & Period    \\ 
\hline
$3$  & $1/12$                          & $1/2$                         & $3$       \\[0.4em] 

$4$  & $1/72$                          & $1/6$                         & $6$       \\[0.4em] 

$5$  & $13/(2^6 \cdot 3^3 \cdot 5)$                       & $13/(2^4\cdot 3^3)$                      & $30$      \\[0.4em] 

$6$  & $59/(2^9\cdot 3^3 \cdot 5^2)$                     & $59/(2^8\cdot 3^2 \cdot 5)$                    & $60$      \\[0.4em] 

$7$  & $231349/(2^{13}\cdot 3^7 \cdot 5^3 \cdot 7)$            & $231349/(2^{12}\cdot 3^6 \cdot 5^3)$            & $2^3 \cdot 3^2 \cdot 5\cdot 7$    \\[0.4em] 

$8$  & $(11 \cdot 29 \cdot 383)/(2^{14} \cdot 3^3 \cdot 5^4 \cdot 7^3)$            & $(11 \cdot 29 \cdot 383)/(2^{11} \cdot 3^3 \cdot 5^4 \cdot 7^2)$           & $2^3 \cdot 3^2 \cdot 5 \cdot 7$    \\[0.4em] 

$9$  & $(115837\cdot30622157)/$                & $(115837\cdot30622157)/$              & $2^4 \cdot 3^2 \cdot 5\cdot 7 \cdot 11$   \\ 
     & $(2^{21} \cdot 3^9 \cdot 5^5 \cdot 7^4 \cdot 11^2)$          & $(2^{18} \cdot 3^7 \cdot 5^5 \cdot 7^4 \cdot 11^2)$          &           \\[0.4em]

$10$ & $(1321 \cdot 58869143 \cdot 1493426677)/$        & $(1321 \cdot 58869143 \cdot 1493426677)/$      & $2^4 \cdot 3^3 \cdot 5\cdot 7 \cdot 11\cdot 13$ \\ 

     & $(2^{25} \cdot 3^{16} \cdot 5^5 \cdot 7^5 \cdot 11^3 \cdot 13^2)$ & $(2^{24} \cdot 3^{14} \cdot 5^4 \cdot 7^5 \cdot 11^3 \cdot 13^2)$ &   \\
\end{tabular}
\medskip
\caption{Data for $H_m(q)$, obtained using \texttt{Normaliz} \cite{normaliz3}.}
\label{tb:oversemigroupconequasidata}
\end{table}

\section*{Appendix:\ Quasipolynomial data}
\label{sec:quasidata}

The number of integer points in the intersection of the group cone $\mathcal{C}(G)$ with the hyperplane $\sum_{i=1}^{|G|-1} x_i = g$ is given by a quasipolynomial $L_G(g)$ of degree $|G| - 2$.  It is known that the leading coefficient of $L_G(g)$ equals the (relative) volume of the intersection $P$ of $\mathcal C(\ZZ_m)$ with $\sum_{i=1}^{|G|-1} x_i = 1$, and the next coefficient (when it is constant) equals half the (relative) surface area of $P$.  
Tables~\ref{tb:groupconequasidata} and~\ref{tb:kunzpolyhedronquasidata} give coefficient and period data on $L_G(g)$ and $p_m(g)$, respectively.  Note that it is proven in~\cite{alhajjarkunz} that the leading coefficient of $L_G(g)$ with $G = \ZZ_m$ equals the leading coefficient of $p_m(g)$.  We also include in the latter table the initial value $N$ of $g$ for which $p_m(g)$ coincides with $N_m(g)$ for all $g \ge N$.  


On the other hand, $H_m(q)$ is a quasipolynomial of degree $m-1$ whose constant leading coefficient is described geometrically in Theorem~\ref{t:leadingcoefficients}.  If $\gcd(m,q) = 1$, $H_m(q)$ coincides with $o(m,q)$, the number of oversemigroups of $\<m, q\>$.  Table~\ref{tb:oversemigroupconequasidata} gives the top two coefficients and period of $H_m(g)$.  
It is important to note that the period of $H_m(q)$ (on all values) may be strictly larger than that of $o(m,q)$ (considered only when $\gcd(m,q) = 1$).  


Examining the data in Tables~\ref{tb:groupconequasidata} and~\ref{tb:kunzpolyhedronquasidata} suggests the following.  

\begin{conj}\label{conj:tabledata}
The following hold.  
\begin{enumerate}[(a)]
\item 
For each $G$ with $|G| \ge 5$, the coefficients of the two highest degree terms of $L_G(g)$ are constant, and their quotient equals $\binom{|G|-1}{2}$.  

\item 
For each $m \ge 4$, the coefficients of the two highest degree terms of $p_m(g)$ are constant, and their quotient equals $2\binom{m-1}{2}$.  

\item 
For each $m \ge 3$, $p_m(g) = N_m(g)$ for $g > \binom{m}{2}$, but $p_m(g) \ne N_m(g)$ for $g = \binom{m}{2}$.  

\end{enumerate}
\end{conj}

\end{document}